\documentclass[12pt]{article}
\usepackage{amsmath,amssymb,amsthm, amsfonts}
\usepackage{hyperref}
\usepackage{graphicx}
\usepackage{url}
\usepackage[mathlines]{lineno}
\usepackage{dsfont} 
\usepackage[dvipsnames]{xcolor}

\usepackage{color}
\definecolor{red}{rgb}{1,0,0}

\definecolor{blue}{rgb}{0,0,1}

\definecolor{green}{rgb}{0,.6,0}

\usepackage{float}

\usepackage{tikz}

\newtheorem{thm}{Theorem}[section]
\newtheorem{cor}[thm]{Corollary}
\newtheorem{lem}[thm]{Lemma}
\newtheorem{prop}[thm]{Proposition}

\theoremstyle{definition}

\theoremstyle{definition}

\theoremstyle{definition}

\theoremstyle{definition}

\theoremstyle{definition}

\newcommand{\binf}{B^{[\infty]}}

\newcommand{\Z}{\operatorname{Z}}

\newcommand{\bit}{\begin{itemize}}
\newcommand{\eit}{\end{itemize}}
\newcommand{\ben}{\begin{enumerate}}
\newcommand{\een}{\end{enumerate}}
\newcommand{\beq}{\begin{equation}}
\newcommand{\eeq}{\end{equation}}
\newcommand{\bea}{\begin{eqnarray}} 
\newcommand{\eea}{\end{eqnarray}}
\newcommand{\bpf}{\begin{proof}}
\newcommand{\epf}{\end{proof}\ms}
\newcommand{\bmt}{\begin{bmatrix}}
\newcommand{\emt}{\end{bmatrix}}
\newcommand{\ms}{\medskip}

\newcommand{\noi}{\noindent}

\newcommand{\beqs}{\begin{equation*}} 
\newcommand{\eeqs}{\end{equation*}}
\newcommand{\beas}{\begin{eqnarray*}}
\newcommand{\eeas}{\end{eqnarray*}}

\newcommand{\sub}[1]{_{(#1)}}

\newcommand{\calf}{\mathcal{F}}

\newcommand{\fs}{\rightarrow}

\title{Generalizations of leaky forcing}

\author{Joseph S. Alameda\thanks{Dept.~of Mathematics, Iowa State University, Ames, IA, USA (jalameda@iastate.edu)} \and 
J\"urgen Kritschgau\thanks{Dept.~of Mathematics, Iowa State University, Ames, IA, USA (jkritsch@iastate.edu) Research is supported by NSF grant DMS-1839918} 
\and  Michael Young\thanks{
 Dept.~of Mathematics, Iowa State University, Ames, IA 50011, USA (myoung@iastate.edu) Research is supported by NSF Grant DMS-1719841.} }

\begin{document}

\maketitle


\begin{abstract}
Vertex leaky forcing was recently introduced as a new variation of zero forcing in order to show how vertex leaks can disrupt the zero forcing process in a graph. An edge leak is an edge that is not allowed to be forced across during the zero forcing process.  The $\ell$-edge-leaky forcing number of a graph is the size of a smallest zero forcing set that can force the graph blue despite $\ell$ edge leaks. This paper contains an  analysis of the effect of edge leaks on the zero forcing process instead of vertex leaks. Furthermore,  specified $\ell$-leaky forcing is introduced. The main result is that  $\ell$-leaky forcing, $\ell$-edge-leaky forcing, and specified $\ell$-leaky forcing are equivalent. Furthermore, all of these different kinds of leaks can be mixed so that vertex leaks, edge leaks, and specified leaks are used. This mixed $\ell$-leaky forcing number is also the same as the (vertex) $\ell$-leaky forcing number.
\end{abstract}

\noi {\bf Keywords} zero forcing, leaky forcing, color change rule

\noi{\bf AMS subject classification} 05C57, 05C15, 05C50

\section{Introduction}

Zero forcing was introduced by the AIM Minimum Rank and Special Graphs Work Group in \cite{AIM}, in order to find upper bounds for the maximum nullity for the family of real symmetric matrices whose off-diagonal entries are described by a graph. The zero forcing process uses a set of blue vertices in a graph that color other vertices blue given a color change rule. Given a graph $G$ and a blue vertex $v\in V(G)$, if $v$ has one white neighbor $w$, then $v$ forces $w$ ($v$ colors $w$ blue). Formally, this process is known as the \textit{zero forcing color-change rule}. A \textit{zero forcing set} for $G$ is an initial set of blue vertices $B$ such that after iteratively and exhaustively applying the zero forcing color-change rule, every vertex in $G$ is blue. The \textit{zero forcing number} of a graph is the size of a minimum zero forcing set, and is denoted $\Z(G)$.

Zero forcing has shown up as a way to control quantum systems \cite{BG,S}. In fact, it was shown that if a set of vertices is a zero forcing set, then the associated quantum system is controllable \cite{BDHSY}. Another system that utilizes the zero forcing process is the electric power system. In \cite{HHHH}, Haynes et al. looked into the problem of monitoring an electric power system by placing as few measurement devices as possible. These applications of zero forcing lead to a natural questions: What if something breaks in the system? Is there a way to keep control? These questions were the main focus in \cite{AKWY} and \cite{DK}. In \cite{DK}, Dillman and Kenter introduced \textit{leaky forcing}, which is a variation on zero forcing that focuses on when vertices in a graph are not able to force. Leaky forcing uses the same color-change rule as zero forcing, but certain vertices are not allowed to perform forces. 

Given a graph $G$, a  \textit{vertex leak} (also referred to as a leak) is a vertex in $G$ that is not able to perform a force. An \textit{$\ell$-leaky forcing set}, is a zero forcing set such that for any set of $\ell$ vertex leaks in $G$, exhaustively applying the color-change rule results in every vertex in $G$ becoming blue. The \textit{$\ell$-leaky forcing number} for a graph $G$ is the size of a minimum $\ell$-leaky forcing set, and is denoted by $\Z\sub{\ell}(G)$. Notice that $\Z(G)=\Z\sub0(G)$. Furthermore, the notion of how resilient a graph is to leaks, and which structures need to be circumvented in a graph for a zero forcing set to be an $\ell$-leaky forcing set were explored in \cite{AKWY}. The notation used in this paper will follow the notation introduced in \cite{AKWY}.  The rest of this section contains results from \cite{AKWY} which are useful for exploring variations of leaky forcing.

In general, let $B\subseteq V(G)$ be an initial set of blue vertices in $G$. If vertex $u$ colors $v$ blue, then $u$ forces $v$ and denote it by $u\fs v$. The symbol $u\fs v$ is called a force. 
A \emph{set of forces $F$ of $B$ in $G$} is a set of forces such that there is a chronological ordering of the forces in $F$ where each force is valid and the whole graph turns blue. When the set $B$ is clear form context, $F$ may be referred to as a forcing process of $B$ or a forcing process $F$ (suppressing the reference to $B$).
Intuitively, $F$ represents the instructions for how $B$ can force $G$ blue, or provides a proof that $B$ is a zero forcing set. 
Implicitly, $F$ gives rise to discrete time steps in which sets of white vertices turn blue. A set $B'$ such that $B\subseteq B'\subseteq V(G)$ is \emph{obtained from $B$ using $F$} if $B$ can color $B'$ blue using only a subset of forces in a forcing process $F$. More generally, $B'$ is obtained from $B$ if there is some forcing process $F$ by which $B$ can color $B'$ blue.

The set $B^{[\infty]}$ is the set of blue vertices after the zero forcing rule has been exhaustively applied with $B$ as an initial blue set. Furthermore, $\binf_L$ will be determined after a set of leaks $L$ has been chosen. In particular, $\binf_L$ is the set of blue vertices obtained from $B$ with leaks $L$ after the zero forcing rule has been exhaustively applied.
Let $\calf (B)$ denote the set of all possible forces given a vertex set $B$. That is, $u\fs v\in \calf(B)$ if there exists a set of forces $F$ of $B$ in $G$ that contains $u\fs v$.  Given this notation, $B$ is an $\ell$-leaky forcing set if for every $L\subseteq V(G)$ with $|L|=\ell$ there exists a forcing process $F$ such that if $u\fs v\in F$, then $u\notin L$. 

Suppose $S\subseteq V(G)$ and $F$ is a forcing process. Let \[F(S)=\{ x\fs y\in F: y \notin S\}.\] By extension, \[F\setminus F(S)=\{ x\fs y\in F: y\in S\}.\] The following lemma proves that abandoning process $F$ to follow process $F'$ creates a new forcing process. 

\begin{lem}\label{switch}\cite{AKWY}
Let $B$ be a blue set in $G$ with forcing processes $F$ and $F'$. Then $(F\setminus F(B'))\cup F'(B')$ is a forcing process of $B$ for any $B'$ obtained from $B$ using $F.$
\end{lem}

The next lemma shows that for any $(\ell-1)$-leaky forcing set $B$ and set of $\ell$ vertex leaks $L$, there exists a time when all $\ell$ leaks in $L$ are blue. Furthermore, there is also a time when all but one of the $\ell$ leaks in $L$ are blue. 

\begin{lem}\label{fail}\cite{AKWY}
If $B$ is an $(\ell-1)$-leaky forcing set and $L$ is a set of $k\geq \ell$ vertex leaks, then $|L\setminus B^{[\infty]}_L|\leq k-\ell$.
\end{lem}

The previous two lemmas are used to prove Theorem \ref{setofforces}. The gist of the proof is to use a forcing process that turns all but one of the leaks blue. This is possible by Lemma \ref{fail}. At this point, the forcing process is abandoned for a process that will  completely force the graph despite the remaining leak. Switching forcing processes is justified by Lemma \ref{switch}.

\begin{thm}\label{setofforces}\cite{AKWY}
A set $B$ is an $\ell$-leaky forcing set if and only if $B$ is an $(\ell-1)$-leaky forcing set such that for every set of $\ell -1$ vertex leaks $L$ and $v\in V(G)\setminus B$ there exists $x\fs v , y\fs v\in \mathcal F_L(B)$ with $y\neq x$.
\end{thm}

Edge-leaky forcing is introduced in Section \ref{edge}. The main result of this section is that the $\ell$-edge-leaky  forcing number is the same as the  $\ell$-leaky forcing. Section \ref{spec} introduces specified leaks, and shows that preventing a directional force is equivalent also equivalent to vertex leaky forcing. In Section \ref{mix},  vertex leaks, edge leaks, and specified leaks are mixed in the leak set without changing the underlying behavior of leaky forcing. Furthermore, in Section \ref{secind}, sets of leaks with a particular underlying structure are explored.  In general, analogs of Lemma \ref{fail} will be used to conclude that the condition in Theorem \ref{setofforces} applies for $\ell$-edge-leaky forcing and specified $\ell$-leaky forcing.

\section{On edge-leaky forcing}\label{edge}

A natural generalization of $\ell$-leaky forcing is to consider what happens when  forces are prohibited from passing over particular edges. An edge $xy$ is an \emph{edge leak} if neither $x\fs y$ nor $y\fs x$ are allowed. A set of blue vertices $B$ is an $\ell$-edge-leaky forcing set if $B$ can turn the whole graph $G$ blue given any set of $\ell$ edge leaks. Denote the $\ell$-edge-leaky forcing number of a graph $G$ by $\Z'\sub\ell(G)$.  Setting both $x$ and $y$ as vertex leaks is a strictly stronger constraint on the zero forcing process than setting $xy$ as an edge leak. However, setting $x$ as a vertex leak is not obviously as  strong as setting $xy$ as an edge leak, since setting $x$ as a vertex leak still allows $y\fs x$. This makes the following result somewhat surprising.

\begin{thm}\label{edgedouble}
A set $B$ is a  $1$-edge-leaky forcing set if and only if for all $v\in V(G)\setminus B$, there exists $x\fs v,y\fs v\in \calf(B)$ with $y\neq x$.
\end{thm}

\begin{proof}
Assume that $B$ is a $1$-edge-leaky forcing set. This implies that $B$ is a zero forcing set with forcing process $F$. Let $v\in V(G)\setminus B$ and $x\fs v\in F$. Since $B$ is a $1$-edge-leaky forcing set, there exists a forcing process $F'$ by which $B$ turns $G$ blue despite setting $xv$ as an edge leak. Therefore, $F'$ must contain a force $y\fs v$ where $y\neq x$. Thus $x\fs v,y\fs v\in \calf(B)$, proving the forward direction. 

Assume that for all $v\in V(G)\setminus B$, there exists $x\fs v, y\fs v\in \calf(B)$ with $y\neq x$. Clearly, this implies that $B$ is a zero forcing set of $G$ with forcing process $F$. Let $xv$ be an arbitrary edge leak. If neither $x\fs v$ nor $v\fs x$ are in $F$, then there is nothing to show. Therefore, without loss of generality, assume that $x\fs v\in F$. Let $B'$ be a set of blue vertices obtained from $B$ using $F$ such that $x\fs v$ is valid given $F$ (were it not for $xv$ being an edge leak), and $v\notin B'$. By assumption, there exists $y\fs v\in \calf(B)$ where $y\neq x$. This implies that there exists a set of forces $F'$ of $B$ in $G$ with $y\fs v$. Since $x\in B'$, it follows that $v\fs x\notin F'(B')$. Therefore, $(F\setminus F(B'))\cup F'(B')$ is a forcing process of $B$ that does not use $xv$.
\end{proof}

With Theorem \ref{edgedouble}, its not as surprising that the $\ell$-edge-leaky forcing number is equivalent to the $\ell$-leaky forcing number for all $\ell\geq 0$. The next lemma  finds an appropriate time to switch forcing processes and controls how edge leaks and forcing sets interact. Let $L- S$ where $S\subseteq V(G)$ denote the edges in $L$ that do not have vertices in $S.$ Explicitly, 
\[L- S=\{xy\in L: x,y\notin S\}.\] By extension, \[L\setminus (L- S)=\{ xy\in L: x\in S \text{ or } y\in S\}.\]

\begin{lem}\label{edgefail}
If $B$ is an $(\ell-1)$-edge-leaky forcing set and $L$ is a set of $k\geq \ell$ edge leaks, then $|L- B^{[\infty]}_L|\leq k-\ell$.
\end{lem}

\begin{proof}
Assume that $L$ is a set of $\ell$ edge leaks, and let $|L- B^{[\infty]}_L|\geq k-\ell+1$. 
Furthermore, let $L'=L\setminus (L- B^{[\infty]}_L)$. Since $|L'|=|L|-|L-\binf_L|$, it follows that $|L'|\leq k-k+\ell-1= \ell-1$.
Notice that edge leaks $uv\in L-B^{[\infty]}_L$ did not change the zero forcing behavior of $B$. In particular, these edge leaks never played a role in stopping $B$ from propagating because their endpoints never were forced. Therefore, $L'$ is a set of at most $\ell-1$ edge leaks which shows that $B$ is not an $(\ell-1)$-edge-leaky forcing set.
\end{proof}

As in the vertex leaky setting, Lemma \ref{edgefail} says that if  $B$ is an  $(\ell-1)$-edge-leaky forcing set and $L$ is a set of $\ell$ edge leaks, then $B$ forces at least one vertex in every edge leak. 

\begin{thm}\label{edgesetofforces}
A set $B$ is an $\ell$-edge-leaky forcing set if and only if $B$ is an $(\ell-1)$-leaky forcing set such that for every set of $\ell-1$ vertex leaks $L$ and $v\in V(G)\setminus B$, there exists $x\fs v,y\fs v\in \calf_L(B)$ with $y\neq x$. 
\end{thm}

\begin{proof} 
Proceed by induction on $\ell$. Notice Theorem \ref{edgedouble} is the base case when $\ell=1$. Assume that the claim holds for all $r<\ell$. 

Let $B$ be an $(\ell-1)$-leaky forcing set such that for every set of $\ell-1$ vertex leaks $L$ and $v\in V(G)\setminus B$, there exists $x\fs v, y\fs v\in \calf_L(B).$ Clearly, $B$ is an $(\ell-2)$-leaky forcing set such that for every set of $\ell-2$ vertex leaks $L$ and $v\in V(G)\setminus B$, there exists $x\fs v, y\fs v\in \calf_L(B)$. Therefore, by the induction hypothesis, $B$ is an $(\ell-1)$-edge-leaky forcing set. 

Let $L$ be a set of $\ell$ edge leaks. By Lemma \ref{edgefail}, it is possible to apply forces one by one until every edge in $L$  contains a blue vertex. Let  $B'$ be the resulting set of blue vertices.  Notice that $B'$ is an $(\ell-1)$-edge-leaky set since $B\subseteq B'$. Therefore, if $B'$ contains an edge in $L$, then there is nothing left to show. Thus, assume that every edge in $L$ contains at exactly one blue vertex in $B'$. 

Let $A\subseteq \{x\in B': xy \in L\}$  such that $|L-A|\leq 1$ and $|A|\leq \ell-1$. Notice that vertices in $A$ can only perform a force if an edge in $L$ is entirely blue first. Therefore, assume  that vertices in $A$ never perform a force;  otherwise, there is nothing left to show.

Let $G^*= G-A$. Since $B$ is an $(\ell-1)$-leaky forcing set, it follows that  $B^*=B'\setminus A$ is a zero forcing set of $G^*$. At this point there is at most one edge leak from $L$ in $G^*$. Let $v\in V(G^*)\setminus B^*$. By assumption, there exists $x\fs v,y\fs v\in\calf_A(B)$ in $G$ with $y\neq x$. Notice that $x,y\notin A$ since $v$ is a white neighbor of both $x$ and $y$. Therefore, $x,y\in V(G^*)$ and $x\fs v,y\fs v\in \calf_A(B^*)$. By Theorem \ref{edgedouble}, $B^*$ is a $1$-edge-leaky forcing set of $G^*$. Thus, $B^*$ can color $G^*$ blue, demonstrating that $B$ is an $\ell$-edge-leaky forcing set. 

To prove the contrapositive of the forward direction, assume that $B$ is an $(\ell-1)$-leaky forcing set, $L=\{x_1,\dots, x_{\ell-1}\}$ is a set of leaks, and $v\in V(G)\setminus B$ such that if $x\fs v, y\fs v \in \calf_L(B)$, then $x=y$. Since $B$ is an $(\ell-1)$-leaky forcing set, there exist $x_0\fs v\in \calf_L(B)$. Let $L'=L\cup \{x_0\}$ be a set of $\ell$ vertex leaks and exhaustively apply the zero forcing rule so that $B_{L'}^{[\infty]}$ is blue.  Notice that $B$ is not an $\ell$-leaky forcing set by Theorem \ref{setofforces}; so   $B_{L'}^{[\infty]}\subset V(G)$ (strictly contained). However, by Lemma \ref{fail},  $L'\subseteq B_{L'}^{[\infty]}$.

To complete the proof, set of vertex leaks $L'$ will be converted into a set of edge leaks. Notice that every vertex $x_i\in L'$ has exactly one white neighbor $y_i$; otherwise, it is possible to remove a vertex from $L'$ to conclude that $B$ is not an $(\ell-1)$-leaky forcing set. Let $L^*=\{x_iy_i: 0\leq i \leq \ell-1\}$. Now $L^*$ demonstrates that $B$ is not an $\ell$-edge-leaky forcing set.
\end{proof}

\begin{cor}\label{edgefund}
For any graph $G$ and $\ell\geq 0$,
\[\Z\sub \ell(G)=\Z'\sub\ell(G).\]
In particular, $B$ is an $\ell$-leaky forcing set if and only if $B$ is an $\ell$-edge-leaky forcing set. 
\end{cor}

The combination of Theorems \ref{setofforces} and \ref{edgesetofforces} provide insight into how leaks interact with the zero forcing rule. In particular, vertex leaks are nicer than edge leaks. Once a vertex leak turns blue, it can safely be deleted from the graph and disregarded for the rest of the process. Edge leaks do not afford us the same luxury. Even if an endpoint of an edge leak turns blue, the vertex cannot be deleted without further care, since it might perform a force later.

\section{On specified-leaky forcing}\label{spec}

Throughout this section,  $v\fs u$ is a \emph{specified leak} if  $v$ is prohibited from forcing $u$.  In this sense, setting a vertex $v$ as a leak represents the set of specified leaks $\{v\fs u: u\in N(v)\}$, and setting an edge $uv$ as a leak represents the set of specified leaks $\{v\fs u, u\fs v\}$. 

It seems as though prohibiting $v\fs x$ and $v\fs y$ is not more restrictive than prohibiting just $v\fs x$ or $v\fs y$, but not both. This is more intuitive after considering the fact that in any particular forcing process $F$, only setting $v\fs x$ or $v\fs y$ as a specified leak poses a problem since $v\fs x$ and $v\fs y$ are not both in $F$. Furthermore, the strength of leaks being picked after the initial blue sets makes a single leak $v\fs x$ as devastating as two leaks $v\fs x, v\fs y$. 

To formalize this intuition a little more, consider the following definitions. A set $B$ is a \emph{specified $\ell$-leaky forcing set of $G$} if $B$ can color $G$ blue when any set of $\ell$ forces are prohibited. Let $\Z^s\sub \ell(G)$ be the minimum size of a specified $\ell$-leaky forcing set of $G$. 

In Section \ref{edge}, Lemma \ref{edgefail} is used to control the interaction between an initial blue set and a set of edge leaks.  However, the proof of  Theorem \ref{specsetofforces} does not require a lemma analogous to Lemma \ref{edgefail} even thought the statements of the two theorems are similar. Proposition \ref{specfail} is analogous to Lemma \ref{edgefail}, and will be proven at the end of this section.

Consider the following definitions before proceeding with the proof of Theorem \ref{specsetofforces}:  If $v\fs u$ is a specified leak, the $v$ is called \emph{the tail of the leak $v\fs u$} and $u$ is the \emph{head of the leak $v\fs u$}. Let $T(L)=\{ x: x\fs y\in L\}$ be the set of tails of $L$ and $H(L)=\{ y:x\fs y\in L\}$ be the set of heads of $L$.

\begin{thm}\label{specsetofforces}
 A set $B$ is a specified $\ell$-leaky forcing set if and only if $B$ is an $(\ell-1)$-leaky forcing set such that for every set of $\ell-1$ vertex leaks $L$ and $v\in V(G)\setminus B$, there exist $x\fs v,y\fs v\in \calf_L(B)$ with $y\neq x$.
\end{thm}

\begin{proof}
Let $L$ be a set of $\ell$ specified leaks that  shows that $B$ is not a specified $\ell$-leaky forcing set. Notice that $|T(L)|\leq \ell$. Therefore, $T(L)$ demonstrates that $B$ is not an $\ell$-leaky forcing set. Thus, by Theorem \ref{setofforces}, there exist $v\in V(G)\setminus B$ such that if $x\fs v,y\fs v\in \calf_{T(L)}(B)$ then $y=x$. 

Suppose that $B$ is an $(\ell-1)$-leaky forcing set, $L=\{x_1,\dots, x_{\ell-1}\}$ is a set of $\ell-1$ leaks, and $v_0\in V(G)\setminus B$ such  $x\fs v_0, y\fs v_0\in \calf(L)$ implies $y=x$. Since $B$ is an $(\ell-1)$-leaky forcing set, there exists $x_0\fs v_0\in \calf_L(B)$.  Let $L'=L\cup \{x_0\}$. By Lemma \ref{fail},
\[|L'\setminus \binf_{L'}|=0.\] Notice that if $|N(y)\cap \binf_{L'}|\geq 2$ for some $y\in L'$, then $L'\setminus\{y\}$ would show that $B$ is not an $(\ell-1)$-leaky forcing set. Therefore, each vertex  $x_i\in L'$ has exactly one white neighbor $v_i\in V(G)\setminus \binf_{L'}$. The set  of specified leaks 
\[\{x_i\fs v_i: 0\leq i\leq \ell-1\}\] shows that $B$ is not a specified $\ell$-leaky forcing set. 
\end{proof}

\begin{cor}\label{specfund}
For any graph $G$ and $\ell\geq 0,$
\[ \Z^s\sub\ell(G)=\Z\sub\ell(G).\]
In particular, $B$ is an $\ell$-leaky forcing set if and only if $B$ is a specified $\ell$-leaky forcing set. 
\end{cor}

As previously noted, Proposition \ref{specfail} controls the interaction between a specified $(\ell-1)$-leaky forcing set and a set of specified leaks $L$. 
Suppose that $L$ is a set of specified leaks and let $S\subseteq V(G)$.
Let \[ L-S=\{ x\fs y\in L: x\notin S\}\]
and \[L\setminus (L-S)=\{x\fs y\in L: x\in S\}.\] The next Proposition controls how many leaks are required to halt a specified $(\ell-1)$-leaky forcing set. This is an analogous result to Lemma \ref{edgefail}.

\begin{prop}\label{specfail}
If $B$ is an $(\ell-1)$-leaky forcing set and $L$ is a set of $k\geq \ell$ specified leaks, then $|L-B^{[\infty]}_L|\leq k-\ell$. 
\end{prop}

\begin{proof}
Assume that $L$ is a set of $k\geq \ell$ specified leaks, and let $|L-B^{[\infty]}_L|\geq k-\ell+1$. If $v\in B^{[\infty]}_L$ and $v$ has exactly one white neighbor $u$, then $v\fs u\in L$ by the leaky forcing rule. Furthermore, let $L'=L\setminus(L-B^{[\infty]}_L)$ and notice that $|L'|\leq \ell-1$. Specified leaks $x\fs y\in L-B^{[\infty]}_L$ did  not change the zero forcing behavior of $B$. In particular, these specified leaks never played a role in stopping $B$ from propagating because the tail never turned blue. Therefore, $T(L')$ is a set of at most $\ell-1$  vertex leaks which show that $B$ is not an $(\ell-1)$-leaky forcing set. 
\end{proof}

\section{On mixed-leaky forcing} \label{mix}

In this section, investigates what happens when a system has various types of leaks preventing the zero forcing process from finishing. A set $B\subseteq V(G)$ is a \textit{mixed $\ell$-leaky forcing set} of a graph $G$ if $B$ can color $G$ blue despite any set of $\ell$ vertex leaks, edge leaks, or specified leaks (refer to these collectively as \emph{leaks}). Let $\Z\sub{\ell}^m(G)$ be the minimum size of a mixed $\ell$-leaky forcing set.

\begin{lem}\label{mixfail}
Let  $L=L_1\cup L_2\cup L_3$ be a set of $k\geq \ell$ leaks where $L_1$ is the set of vertex leaks, $L_2$ is the set of edge leaks, and $L_3$ is the set of specified leaks.
If $B$ is a mixed $(\ell-1)$-leaky forcing set, then $|L_1\setminus B^{[\infty]}_L|+ |L_2- B^{[\infty]}_L| + |L_3- B^{[\infty]}_L| \leq k-\ell.$ 
\end{lem}

\begin{proof}
To prove the contrapositive, assume that $|L_1\setminus B^{[\infty]}_L|+ |L_2- B^{[\infty]}_L| + |L_3- B^{[\infty]}_L| \geq k-\ell+1.$ Every vertex in $B^{[\infty]}_L$ has either $0$, $1$, or at least $2$ white neighbors. If $v\in B^{[\infty]}_L$ such that $v$ has exactly one white neighbor $u$, then either $v\in L_1$, $vu\in L_2$, or $v\fs u\in L_3$. Let $L'=[L_1\setminus (L_1\setminus B^{[\infty]}_L)]\cup[ L_2 \setminus (L_2- B^{[\infty]}_L)] \cup [L_3 \setminus (L_3- B^{[\infty]}_L)].$ Since this is a disjoint union, 
\begin{align*}
    |L'| &= |L_1\setminus (L_1\setminus B^{[\infty]}_L)| + |L_2 \setminus (L_2- B^{[\infty]}_L)| + |L_3 \setminus (L_3- B^{[\infty]}_L)| \\
    &\leq \ell-1.
\end{align*}
Notice that any leak in either $L_1\setminus B^{[\infty]}_L$, $L_2- B^{[\infty]}_L$, or $L_3- B^{[\infty]}_L$ did not change the zero forcing behavior of $B$. In particular, these leaks never played a role in stopping $B$ from propagating because the vertex leaks were never forced blue, the tails of the specified leaks were never forced blue, and the endpoints of the edge leaks were never forced blue. Therefore, $L'$ is a set of at most $\ell-1$ leaks which shows $B$ is not a mixed $(\ell-1)$-leaky forcing set. 
\end{proof}

Notice that if $L_2$ and $L_3$ are empty, then  Lemma \ref{fail} is recovered. With this more general formulation of leaky forcing,  the next theorem can be proven.

\begin{thm}
A set $B$ is a mixed $\ell$-leaky forcing set if and only if $B$ is an $(\ell-1)$-leaky forcing set such that for every $\ell-1$ vertex leaks $L$ and $v\in V(G)\setminus B$, there exist $x\fs v,y\fs v\in \calf_L(B)$ with $y\neq x$.
\end{thm}

\begin{proof}
Proceed by induction on $\ell$. Notice either Theorem \ref{setofforces}, Theorem \ref{edgedouble}, or Theorem \ref{specsetofforces} handles the base case when $\ell=1.$ Assume the claim holds for all $r<\ell$.

Let $B$ be an $(\ell-1)$-leaky forcing set such that for every set of $\ell-1$ vertex leaks $L$ and $v\in V(G)$, there exists $x\fs v, y\fs v \in \calf_L(B)$ with $y\neq x$. Clearly, $B$ is an $(\ell-2)$-leaky forcing set such that for every set of $\ell-2$ vertex leaks $L$ and $v\in V(G)\setminus B$, there exists $x\fs v, y\fs v \in \calf_L(B)$ with $y\neq x$. Therefore by the induction hypothesis, $B$ is a mixed $(\ell-1)$-leaky forcing set.

Let $L=L_1\cup L_2\cup L_3$ be a set of $\ell$ leaks where $L_1$ is the set of vertex leaks, $L_2$ is the set of edge leaks, and $L_3$ is the set of specified leaks. By Lemma \ref{mixfail}, it is possible to apply forces one by one until every vertex leak in $L_1$ is blue, every edge leak in $L_2$ contains a blue vertex, and the tails of  specified leaks in $L_3$ are blue. Let $B'$ be the resulting set of blue vertices. Notice that $B'$ is a mixed $(\ell-1)$-leaky forcing set since $B\subseteq B'.$ If $B'$ contains an edge from either $L_2$ or $L_3$, then there is nothing left to show. Therefore, assume that the edges in $L_2$ are incident to one blue vertex, and only the tails of forces in $L_3$ are blue. 

Let $L'\subset L$ be a set of $\ell-1$ leaks. Notice that blue vertices in $L_1\cap L'$,  blue vertices incident to edges in $L_2\cap L'$, and the tails of specified leaks in $L_3\cap L'$ can be deleted.  Let $A$ be the set of these vertices. 

Consider $G^*=G-A$. Since $B$ is an $(\ell-1)$-leaky forcing set, it follows that $B^*=B'\setminus A$ is a zero forcing set for $G^*$. At this point there is at most one leak from $L$ in $G^*$. Let $v\in V(G^*)\setminus B^*$. By assumption there exists $x\fs v, y\fs v \in \calf_A(B)$ in $G$ with $y\neq x$. Since $x,y\notin A$, it follows that $x,y \in V(G^*)$ and $x\fs v, y\fs v \in \calf(B^*).$ By Theorems \ref{setofforces}, \ref{edgesetofforces}, and \ref{specsetofforces}, $B^*$ is a $1$-edge-leaky forcing set, a $1$-leaky forcing set, and a specified $1$-leaky forcing set. Thus, $B^*$ can color $G^*$ blue, demonstrating that $B$ is a mixed $\ell$-leaky forcing set. 

To prove the contrapositive of the forward direction, assume $B$ is an $(\ell-1)$-leaky forcing set, $L=\{x_1,\dots,x_{\ell-1}\}$ is a set of vertex leaks and $v\in V(G)\setminus B$ such that if $x\fs v, y\fs v \in \calf_L(B)$, then $x=y$. Since $B$ is an $(\ell-1)$-leaky forcing set, there exists $x_0\fs v\in \calf_L(B)$. Let $L'=L\cup \{x_0\}$ be a set of $\ell$ vertex leaks. Notice that $L'$ demonstrates that $B$ is not an $\ell$-leaky forcing set. Thus, $B$ is also not a mixed $\ell$-leaky forcing set.
\end{proof}

\begin{cor}
For any graph $G$ and $\ell\geq 0$, 
\[\Z_{(\ell)}(G)=\Z_{(\ell)}^m(G).\]
In particular, $B$ is an $\ell$-leaky forcing set if and only if $B$ is a mixed $\ell$-leaky forcing set.
\end{cor}

\section{Independent sets of specified leaks} \label{secind}

Corollaries \ref{edgefund} and \ref{specfund} suggest that the strength of a set of specified leaks is somewhat independent of the number of leaks or their relative layout within the graph. In particular, a set of $\ell$ vertex leaks or $\ell$ edge leaks is at most as strong as a set of $\ell$ specified leaks even though $\ell$ vertex or edge leaks corresponds to more than $\ell$ specified leaks. Thus, arranging specified leaks into sets of out-stars or $2$-cycles is in some sense inefficient. The goal of this section is to formally develop what it means for the shape of a set of specified leaks to be irrelevant. In particular,  arranging specified leaks into out-stars or $2$-cycles is not an efficient use of leaks. 

Let $L$ be a set of specified leaks on $V(G)$. Notice that a specified leak $v\fs u$ can be thought of as a directed edge from $v$ to $u$. Therefore, $L$ naturally corresponds to the edge set of a directed graph on the vertex set $V(G)$. This gives rise to a notion of isomorphic sets of specified leaks. Let $L_1$ and $L_2$ be sets of specified leaks on $V(G)$. A set of specified leaks $L_1$ is \emph{isomorphic} to $L_2$ if there exists a bijection $\phi: V(G)\to V(G)$ such that $x\fs y\in L_1$ if and only if $\phi(x)\fs \phi(y) \in L_2$. 

A set $B$ is an $L$-leaky forcing set if $B$ can turn $G$ blue despite any set of $L_1$ leaks that is isomorphic to $L_2$ where $L_2\subseteq L$. Correspondingly, the $L$-leaky forcing number of $G$, denoted $\Z\sub L(G)$, is the size of the smallest $L$-leaky forcing set.  It is implicitly assumed that $L$ is a set of specified leaks on $V(G)$.

A set of specified leaks $L$ is a \emph{set of independent leaks} if for all $x\fs y, v\fs u\in L$, it follows that $x\neq v$ and $y\neq v$. Equivalently, $L$ is independent if $|T(L)|=|L|$ and $T(L)\cap H(L)=\varnothing$. Let $I(L)$ denote the size of the largest  set of independent leaks contained by  $L$. 

These definitions let us abstract away from the specific shape of a set of specified leaks and focus on the parameter that seems to matter. In particular, a set of specified leaks $L$ is no stronger than the a maximum set of independent leaks contained in $L$.

\begin{thm}\label{ind}
Let $L$ be a set of specified leaks on $V(G)$ and let $\ell= I(L)$. A specified $\ell$-leaky forcing set $B$ is an $L$-leaky forcing set. That is,
\[\Z\sub L(G)\leq \Z^s\sub\ell(G).\]
\end{thm}

To prove Theorem \ref{ind}, consider \emph{active leaks}. The set of active leaks given a blue set $B$ and a set of specified leaks $L$ is the set of leaks in $L$ that actively prevents $B$ from performing a force. Formally the set of active leaks is given by 
\[A(B,L)=\{x\fs y\in L: x\in B, \{y\}= N(x)\setminus B\}.\]

\begin{proof}
First, consider the contrapositive of the desired result. Suppose that $B$ is not an $L$-leaky forcing set. Therefore, there exists a set of specified leaks $L'$ which is isomorphic to a subset of $L$ that prevents $B$ from coloring all of $G$ blue. Let $\binf_{L'}$ be the set of blue vertices obtained from $B$ given $L'$ by exhaustively applying forces. Notice that $\binf_{L'}\neq V(G)$, and let $A=A(\binf _{L'}, L')$. Since $A$ is a set of independent leaks, it follows that
\[ |A|\leq I(L')\leq I(L).\]
Furthermore, $A$ demonstrates that $B$ is not a specified $\ell$-leaky forcing set.
\end{proof}

The converse of Theorem \ref{ind} holds when $I(L)=1$. 

\begin{prop}\label{inddouble}
Let $L$ be a set of specified leaks on $V(G)$ and let $1= I(L)$. A set $B$ is an $L$-leaky forcing set if and only if $B$ is a specified $1$-leaky forcing set. 
\end{prop}

\begin{proof}
The backward direction Proposition \ref{inddouble} is covered by Theorem \ref{ind}. Therefore, assume that $B$ is not a specified $1$-leaky forcing set. This implies that there exists $L'=\{x\fs y\}$ that stops $B$ from turning $G$ blue. By assumption, $L$ has a set of independent leaks of size $1$. Therefore, $L'$ is isomorphic to a subset of $L$. Therefore, $B$ is not an $L$-leaky forcing set.
\end{proof}

The proof of Proposition \ref{inddouble} relies on the fact that, up to isomorphism, there is only one set of independent leaks. Proving the converse of Theorem \ref{ind} fails since an arbitrary set of $\ell$ independent leaks cannot always be injected into $L$ when $I(L)=\ell$. To illustrate this point, consider the following example. Let $G=K_{\ell+1}\square K_2$, $\ell\geq 2$ with vertex set $V(G)=\{x_1,\dots, x_{\ell+1}, y_1,\dots, y_{\ell+1}\}$ where the sets $\{x_i:1\leq i\leq \ell+1\}, \{y_i: 1\leq i\leq \ell+1\}$ induce cliques, and $\{x_iy_i: 1\leq i\leq \ell+1\}$ induces a matching. Let $L_1=\{x_i\fs y_i: 1\leq i\leq \ell+1\}$, and $L_2=\{x_i\fs x_{\ell+1}: 1\leq i \leq \ell\}$. Suppose that $B=\{x_i: 1\leq i\leq \ell+1\}$. First, notice that $B$ is not a specified $2$-leaky forcing set, since $L=\{x_1\fs y_1, x_2\fs y_2\}$ prevents $B$ from turning $y_1,y_2$ blue. This also shows that $B$ is not an $L_2$-leaky forcing set. However, $B$ is an $L_1$ leaky forcing set. Since $I(L_1)=I(L_2)=\ell$, this example shows that the converse of Theorem \ref{ind} is false for $\ell\geq 2$.

\section{Closing remarks}

Though leaky forcing is a natural generalization of zero forcing, its relationship to the linear algebra roots of zero forcing is less clear. Consider the following system:

\begin{align*}
a_{1,1}x_1+a_{1,2}x_2+a_{1,3}x_3+a_{1,4}x_4&=0\\
a_{2,1}x_1+a_{2,2}x_2+a_{2,3}x_3+0x_4&=0\\
a_{3,1}x_1+a_{3,2}x_2+a_{3,3}x_3+0x_4&=0\\
a_{4,1}x_1+0x_2+0x_3+a_{4,4}x_4&=0
\end{align*}

or equivalently,

\[\begin{pmatrix}
a_{1,1}& a_{1,2}& a_{1,3}&a_{1,4}\\
a_{2,1}& a_{2,2}& a_{2,3}& 0 \\
a_{3,1}& a_{3,2}& a_{3,3}& 0\\
a_{4,1}& 0& 0&a_{4,4}
\end{pmatrix}
\begin{pmatrix}
x_1\\
x_2\\
x_3\\
x_4
\end{pmatrix}=
\begin{pmatrix}
0\\0\\0\\0
\end{pmatrix}.\]
Here it is required that $a_{i,j}=a_{j,i}\neq 0$ for $i\neq j$. Zero forcing studies the minimum number of entries in the $x$ vector that need to be set to $0$ before one can conclude that the whole $x$ vector is identically $0$. In particular, determining that  $x_2=x_4=0$ implies $x_1=x_3=0$ is equivalent to seeing that $\{x_2,x_4\}$ is zero forcing set in Figure \ref{paw}. In this setting $x_4x_1$ as an edge leak corresponds to assuming that $a_{1,4}$ and $a_{4,1}$ are zero divisors. That is, if $a_{4,1}$ is a zero divisor and $x_4=0$, then $a_{4,1}x_1+a_{4,4}x_4=0$ does not imply that $x_1=0$. Equivalently, if $x_1x_4$ is an edge leak, then $x_1$ cannot be used to turn $x_4$ blue.

 \begin{figure}[H]
        \centering{
        \includegraphics[width=0.15\textwidth]{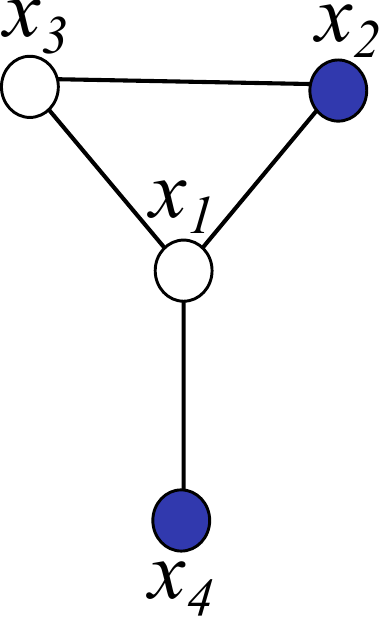}
        \caption{The paw graph with a corresponding zero forcing set.}
        \label{paw}}
    \end{figure}

Let $G$ be an arbitrary graph and $A$ a symmetric matrix with nonzero pattern corresponding to $G$ with $2\ell$ zero divisors in the off-diagonal entries ($\ell$ zero divisors in the upper off-diagonal entries). Under the interpretation in the previous paragraph, $\Z'\sub\ell (G)$ corresponds to the minimum number of $0$ entries in $x$ that force $x=0$ vector under the condition that $Ax=0$.

Up to this point, the authors are unaware of generalizations of the notions dimension, rank, and spectrum for modules over rings with zero divisors and linear transformations thereof. As a general problem, and a curiosity well beyond the scope of this paper, the authors would be very interested to see if the edge-leaky forcing number can be used as a tool to analyze some appropriate notion of minimum rank or maximum nullity for $R$-linear transformations (homomorphisms) of modules $M$  over ring $R$ with zero divisors. 

\subsection*{Acknowledgements}

The authors would like to thank Nathan Warnberg for insightful discussions and feedback. This material is based upon work supported by the National Science Foundation under Grant Numbers DMS-1839918 and DMS-1719841.

\end{document}